\documentclass[a4paper,article] {amsart}
\usepackage{amssymb}
\usepackage{amscd}

\usepackage[leqno]{amsmath}
\usepackage{amsfonts}
\usepackage{amssymb}
\usepackage{amsthm}
\usepackage{amssymb}
\usepackage[all]{xy}
\usepackage{graphicx}
\usepackage{hyperref}
\usepackage{color}

\newcommand{\RR}{\mathbf R}

\newcommand{\ZZ}{\mathbf Z}

\DeclareMathOperator{\esssup}{ess\,sup}
\DeclareMathOperator{\essinf}{ess\,inf}

\theoremstyle{plain}
\newtheorem{theorem}{Theorem}[section]

\newtheorem{lemma}[theorem]{Lemma}

\newtheorem{proposition}{Proposition}[section]

\newtheorem{claim}{Claim}
\theoremstyle{definition}
\newtheorem{definition}[theorem]{Definition}
\newtheorem{example}[theorem]{Example}
\newtheorem{remark}[theorem]{Remark}
\newtheorem{notation}[theorem]{Notation}

\numberwithin{equation}{section}

\renewcommand{\leq}{\leqslant}
\renewcommand{\geq}{\geqslant}

\title[Strain Induced Slowdown of Front Propagation]{Strain Induced Slowdown of Front Propagation in Random Shear Flow via Analysis of G-equations}



\author[Hongwei Gao]{Hongwei Gao}

\address{
Department of Mathematics \\ 
University of California, Irvine   \\ 
CA, 92697\\
USA}
\email{hongweig@math.uci.edu}

\thanks{Partially supported by DMS-1151919} 

\begin{document}

\vspace{18mm} \setcounter{page}{1} \thispagestyle{empty}

\begin{abstract}
It is proved that for the 2-dimensional case with random shear flow of the G-equation model with strain term, the strain term reduces the front propagation. Also an improvement of the main result by Armstrong-Souganidis is provided.
\end{abstract}

\maketitle


\section{Introduction}\label{intro}

\subsection{G-equations and strain effect} The G-equation is a well known model in the study of
turbulent combustion. It is the level set formulation of interface motion laws in the thin interface regime. In the simplest model of the G-equation, the normal velocity of the interface equals a positive constant $s_{L}$ (which is called the laminar speed) plus the normal projection of the fluid velocity $V(x)$, which gives the inviscid G-equation (in the general dimensional situation, we use $x$ as spatial variable and $t$ as time variable):
\begin{equation}
\begin{cases}
G_{t}+V(x)\cdot DG+s_{L}|DG|=0  & (x,t)\in\RR^{d}\times(0,\infty)\\
G(x,0)=G_{0}(x) & x\in\RR^{d}
\end{cases}
\end{equation}

In reality, inter-facial fluctuations appear in front propagation. So, there is a family of G-equations with different oscillation scales.
\begin{equation}
\begin{cases}
G_{t}^{\epsilon}+V(\frac{x}{\epsilon})\cdot DG^{\epsilon}+s_{L}|DG^{\epsilon}|=0 & (x,t)\in\RR^{d}\times(0,\infty)\\
G^{\epsilon}(x,0)=G_{0}(x) & x\in\RR^{d}
\end{cases}
\end{equation}

When $V(x)$ is $\ZZ^{d}$-periodic and nearly compressible, Xin-Yu [\ref{ref J.Xin and Y.Yu 2010}] and Cardaliaguet-Nolen-Souganidis [\ref{ref CNS}] independently proved that $G^{\epsilon}(x,t)\rightarrow \overline{G}(x,t)$ locally uniformly and $\overline{G}$ solves the homogenized Hamilton-Jacobi equation:
\begin{equation}
\begin{cases}
\overline{G}_{t}+\overline{H}(D\overline{G})=0 & (x,t)\in\RR^{d}\times(0,\infty) \\
\overline{G}(x,0)=G_{0}(x) & x\in\RR^{d}
\end{cases}
\end{equation}
The effective Hamiltonian $\overline{H}$ is called the turbulent flame speed or the turbulent burning
velocity in combustion literature. For the stationary ergodic divergence-free flow $V(x,\omega)$  which is the gradient of a stream function that satisfies some integrability condition, Nolen-Novikov [\ref{Nolen and Novikov}] first proved homogenization for the 2-dimensional case. Then for the general dimensional case, if $V(x,\omega)$ is divergence-free and has appropriately small mean, Cardaliaguet-Souganidis [\ref{ref CS random}] proved the homogenization.

Since the flow will stretch or compress the front flame surface, the reaction over the flame front will be affected. Thus the laminar speed $s_{L}$ depends on the flame stretch and therefore can not be constant. To model the strain effect in the G-equation, people extend $s_{L}$ to $s_{L}+c\vec{n}\cdot DV \cdot \vec{n}$, where $\vec{n}$ represents the normal direction. Here the Markstein length $c$ is proportional to the flame thickness. Hence the induced strain G-equation is [\ref{ref N.Peter},\ref{Matalon, M.Matkowsky},\ref{Pelce,P.,Clavin,P.}]
\begin{eqnarray}\label{general strain G-eqn}
G_{t}+V(x,\omega)\cdot DG+s_{L}|DG|+c\frac{DG}{|DG|}\cdot S\cdot DG=0 & S:=\frac{DV+(DV)^{\top}}{2}
\end{eqnarray}
Some interesting questions are: 

(1) Can this strain G-equation be homogenized? 

(2) If yes, how does the strain term affect the turbulent flame speed $\overline{H}(p)$?
\begin{remark}
The strain G-equation (\ref{general strain G-eqn}) is highly non-coercive and non-convex, which increase the difficulty of homogenization.
\end{remark}

When $V$ is a 2-d periodic Cellular Flow, Xin-Yu [\ref{ref Xin and Yu}] showed that due to the existence of a strain term, when the flow intensity (magnitude of $V$) is large enough, the effective Hamiltonian becomes zero. This means that under the effect of strain, the flame is quenched when the flow is too strong.

In this short article, we investigate those questions for 2-d random Shear Flows $V(x,\omega)$ in the stationary ergodic setting.

\subsection{2-d random Shear Flows}
For the $2$-d problem, we denote the space variable by $(x,y)\in\RR^{2}$. Without loss of generality, we assume $s_{L}=1$. We study the problem under a random Shear Flow $V=(v(y,\omega),0)$ and assume $v(y,\omega)$ is stationary ergodic (See section 2 for precise definitions).

Let $p=(m,n)$, the cell problem, if it exists, becomes:
\begin{equation*}
\sqrt{(m+G_{x})^{2}+(n+G_{y})^{2}}+v(y,\omega)\cdot(m+G_{x})+c\frac{(m+G_{x})(n+G_{y})v^{\prime}}{\sqrt{(m+G_{x})^{2}+(n+G_{y})^{2}}}=\overline{H}(m,n,c)
\end{equation*}

This can be reduced to a 1-d problem:
\begin{equation}\label{1-d problem}
\sqrt{m^{2}+(n+G_{y})^{2}}+v(y,\omega)\cdot m+c\frac{m(n+G_{y})v'}{\sqrt{m^{2}+(n+G_{y})^{2}}}=\overline{H}(m,n,c)
\end{equation}

So, it suffices to study the following 1-d Hamiltonian. Since the case $m=0$ is trivial, let's fix $m \ne 0$, and we will denote $\overline{H}(m,n,c)$ by $\overline{H}(n,c)$. As a function of $p$, the following Hamiltonian is not convex but it is level-set convex.
\begin{equation}\label{1-d Hamiltonian}
H(p,x,\omega,c)=\sqrt{m^{2}+p^{2}}+v(x,\omega)\cdot m+c\frac{mpv'(x,\omega)}{\sqrt{m^{2}+p^{2}}}
\end{equation}

In fact,
\begin{equation}
\frac{\partial H(p,x,\omega,c)}{\partial p}=\frac{p^{3}+m^{2}p+cv'm^{3}}{(m^{2}+p^{2})^{\frac{3}{2}}}
\end{equation}
For fixed $x,\omega, c$, $\frac{\partial H}{\partial p}=0$ has a unique real root and $\lim\limits_{p\rightarrow -\infty}\frac{\partial H}{\partial p}=-1, \lim\limits_{p\rightarrow +\infty}\frac{\partial H}{\partial p}=1$. Thus the Hamiltonian is level-set convex(see \textbf{(A2)} in section 2). Actually, by the above facts, $H$ is strictly level-set convex. This means that for each fixed $x,\omega,c$, for any $\mu\in\RR$, $\lbrace p:H(p,x,\omega,c)=\mu \rbrace$ has no interior point.

\subsection{Random homogenization of Hamilton-Jacobi equations with level-set convex Hamiltonians}
Armstrong-Souganidis [\ref{ref Armstrong and Souganidis}] proved random homogenization of the Hamilton-Jacobi equations with level-set convex Hamiltonians. In addition to level-set convexity they require more assumptions. Their proofs depend on the existence of a family of auxiliary functions $\Lambda_{\lambda}\in C(\RR\times\RR)$ that are nondecreasing in both of the arguments and satisfying
\begin{equation}\label{additional assumptions(1)}
\text{For all } \mu\neq\nu, \Lambda_{\lambda}(\mu,\nu)<\max \lbrace \mu,\nu \rbrace
\end{equation}
\begin{equation}\label{additional assumptions(2)}
H(\lambda p+(1-\lambda)q,y,\omega)\leq \Lambda_{\lambda}(H(p,y,\omega),H(q,y,\omega))
\end{equation}
where $p,q,y\in\RR^{d}, \omega\in\Omega$, for each $0<\lambda_{1}\leq \lambda_{2}\leq\frac{1}{2}, \Lambda_{\lambda_{1}}\geq \Lambda_{\lambda_{2}}$. However, the existence of $\Lambda_{\lambda}$ is not straightforward.

So, it is not obvious if the Hamiltonian (\ref{1-d Hamiltonian}) satisfies (\ref{additional assumptions(1)}), (\ref{additional assumptions(2)}). However, based on a very simple modification of the method in [\ref{ref Armstrong and Souganidis}], we will show in section \ref{section 2} that (\ref{additional assumptions(1)}), (\ref{additional assumptions(2)}) are not necessary and random homogenization holds for any level-set convex Hamiltonian with general dimension. Actually, to prove the homogenization of 1-d Hamiltonian in section \ref{section 3}, we do not need the result of section \ref{section 2}. In fact, random homogenization of 1-d coercive Hamiltonian has been established by Armstrong-Tran-Yu[\ref{1-d seperable noncovex by ATY}] in separable case and extended by the author[\ref{Homogenization of general coercive H-J by Gao}] to general coercive case. The following is our main result which will be proved
in section \ref{section 3}. Throughout this paper, all solutions of PDEs are interpreted in the viscosity sense[\ref{Evans}].
\begin{theorem}\label{main thm}
For the 2-dimensional case in the the stationary ergodic setting with a shear flow $V=(v(y,\omega),0)$ such that $v(\cdot,\omega)\in  C^{\infty}(\RR)$ and $v(x,\omega), v'(x,\omega)\in L^{\infty}(\RR\times\Omega)$. Then

$(1)$ The G-equation with strain term (\ref{general strain G-eqn}) can be homogenized. 

$(2)$ For any unit vector $p=(m,n)\in\RR^{2}$ and $c>0,$
\begin{equation*}
\overline{H}(p)=\overline{H}(m,n)\geq \overline{H}(p,c)=\overline{H}(m,n,c)\geq |m|+\sup\limits_{(y,\omega)\in\RR\times\Omega}mv(y,\omega)
\end{equation*}

$(3)$ If $\overline{H}(p)>|m|+\sup\limits_{(y,\omega)\in\RR\times\Omega}mv(y,\omega)$ and $\mathbf{E}[v]=0$, then $\overline{H}(p)= \overline{H}(p,c)$ if and only if $mv\equiv 0.$
\end{theorem}

\begin{remark}
Statement (2) means that the strain term reduces the turbulent flame speed. Since $v'$ changes sign, this is not obvious at all. This result is consistent with concensus in combustion literature that the strain rate plays an important role in slowing down or even quenching flame propagation[\ref{ref P.Ronny}]. This fact has been observed by Xin-Yu [\ref{Xin Yu personal communication}] in the periodic setting.
\end{remark}

\section{A remark on homogenization of level-set convex Hamiltonians}\label{section 2}

In this section, we claim that random Hamilton-Jacobi equations with Hamiltonians that are merely level-set convex can be homogenized. Here $(x,t)\in\RR^{d}\times\RR$ is the space-time variable.
\subsection{Assumptions}
Consider the Hamilton-Jacobi equation
\begin{equation*}
\begin{cases}
u_{t}+H(Du,x,\omega)=0 & (x,t)\in\RR^{d}\times (0,\infty)\\
u(x,0)=u_{0}(x) & x\in\RR^{d}
\end{cases}
\end{equation*}

For $H$, we assume:\\
\textbf{(A1)} Stationary Ergodicity: There exists a probability space $\ (\Omega,\mathcal{F},\mathbf{P})$ and a group $\lbrace\tau_{y}\rbrace_{y\in\RR^{d}}$ of  $\mathcal{F}$-measurable, measure-preserving transformations $\tau_{y}:\Omega \rightarrow \Omega$, i.e. for any $x,y\in\RR^{d}$:
\begin{equation*}
\tau_{x+y}=\tau_{x}\circ \tau_{y} \text{ and } \mathbf{P}[\tau_{y}(A)]=\mathbf{P}[A]
\end{equation*}
Ergodicity:
\begin{equation*}
A\in \mathcal{F}, \ \tau_{z}(A)=A \text{ for every }z\in \RR^{d} \Rightarrow \mathbf{P}[A]\in \lbrace 0,1 \rbrace
\end{equation*}
Stationary:
\begin{equation*}
H(p,y,\tau_{z}\omega)=H(p,y+z,\omega)
\end{equation*}
\textbf{(A2)} Level-Set Convexity: For every $(y,\omega)\in \RR^{d}\times \Omega$ and $p,q\in \RR^{d}$.
\begin{equation*}
H(\frac{p+q}{2},y,\omega)\leq \max \lbrace H(p,y,\omega), H(q,y,\omega) \rbrace
\end{equation*}
\textbf{(A3)} Coercivity: 
\begin{equation*}
\lim\limits_{|p|\rightarrow \infty}\essinf\limits\limits_{(y,\omega)\in \RR^{d}\times \Omega}H(p,y,\omega)=+\infty
\end{equation*}
\textbf{(A4)} Boundedness and Uniform Continuity:
\begin{equation*}
\lbrace H(\cdot,\cdot,\omega):\omega \in \Omega \rbrace \text{ is bounded and equicontinuous on } B_{R}\times \RR^{d} \text{ for any } R>0.
\end{equation*}

\subsection{Comparison Principle for Metric Problem}

We adopt the same notations as in [\ref{ref Armstrong and Souganidis}]; by stationary ergodicity, these are independent of random variable $\omega$. So, we suppress the random variable.

\begin{notation}
\begin{equation*}
\mathcal{L}:=\lbrace \text{real-valued global Lipschitz functions in } \RR^{d}\rbrace 
\end{equation*}
\begin{equation*}
\overline{H}_{*}:=\inf\limits_{w\in\mathcal{L}}\esssup\limits\limits_{y\in\RR^{d}}H(Dw,y)
\end{equation*}
\begin{equation*}
\mathcal{S}:=\lbrace w\in \mathcal{L}:\lim\limits_{|y|\rightarrow \infty}\frac{w(y)}{|y|}= 0 \rbrace, \ 
\end{equation*}

\begin{equation*}
\widehat{H}(p):=\inf\limits_{w\in \mathcal{S}} \esssup\limits\limits_{y\in \RR^{d}}H(p+Dw,y)
\end{equation*}

\end{notation}

For fixed $x\in\RR^{d}$ and $\mu\geq \overline{H}_{*}$, we consider the metric problem
\begin{equation}\label{metric eqn}
\begin{cases}
H(p+Dv,y)=\mu & \RR^{d}\diagdown \lbrace x \rbrace \\
v(x)=0 & \\
\end{cases}
\end{equation}

The idea in [\ref{ref Armstrong and Souganidis}] to determine $\overline{H}(p)$ is to homogenize each level set of $H$. The main tool is a comparison principle(Proposition 3.1 of [\ref{ref Armstrong and Souganidis}]) of the metric problem, and the proof of the comparison principle depends on the additional assumptions (\ref{additional assumptions(1)}) and (\ref{additional assumptions(2)}). For general level-set convex Hamiltonians, we cannot prove the same comparison principle.

Since homogenization is closed under uniform limits, the following question arises: can we add a small perturbation to the level-set convex Hamiltonian such that the perturbed Hamiltonian satisfy (\ref{additional assumptions(1)}) and (\ref{additional assumptions(2)}), and then take the limit? This may work for a carefully constructed perturbation, but it does not work if we simply perturb $H(p,x,\omega)$ by $H_{\epsilon}(p,x,\omega):=\epsilon |p|^{2}+H(p,x,\omega)$ as the following simple example shows.

\begin{example}
Consider $d=1$ and $H(p,x,\omega)=H(p)$ defined by
\begin{eqnarray*}
H(p):=
\begin{cases}
-p &  p\in(-\infty,0] \\
0 &  p\in(0,1]\\
-p+1 &  p\in(1,2]\\
p-3 &  p\in (2,+\infty)
\end{cases}\\
H_{\epsilon}(p):=\epsilon |p|^{2}+H(p)
\end{eqnarray*}

Then for $0<\epsilon\ll 1$, $H_{\epsilon}(0)=0, H_{\epsilon}(1)=\epsilon>0, H_{\epsilon}(2)=4\epsilon-1<0$ violates the level-set convexity. So the perturbation may destroy the structure of level-set convexity.
\end{example}
Fortunately, we observe that by the method introduced in [\ref{ref Armstrong and Souganidis}], wherever we need the comparison principle we actually only need the following weak version comparison principle. To prove the weak version comparison principle, level-set convexity is sufficient.

\begin{lemma}[Weak Comparison Principle]
Assume $\widehat{H}(p)<\mu<\nu<+\infty$ and $u$, $-v \in USC(\RR^{d})$ solve the following equation.
\begin{equation}
H(p+Du,y)\leq \mu<\nu \leq H(p+Dv,y) \text{ in } \RR^{d}\diagdown K, \ K\subset \RR^{d} \text{ compact }
\end{equation} 
with
\begin{equation}
\liminf\limits_{|y|\rightarrow \infty}\frac{v(y)}{|y|}\geq 0
\end{equation}
Then
\begin{equation}
\sup\limits_{\RR^{d}}(u-v)=\max\limits_{K}(u-v)
\end{equation}
\end{lemma}

\begin{remark}
Without loss of generality, we only consider the case $p=0$.\\

\textbf{(A4)} and $\mu<\nu$ $\Rightarrow$ by adding a function with arbitrary small gradient to $v$, we can assume without of loss of generality that:
\begin{equation}
A:=\liminf\limits_{|y|\rightarrow \infty}\frac{v(y)}{|y|}>0
\end{equation}

\textbf{(A3)} $\Rightarrow$ $u$ is Lipschitz $\Rightarrow$ $\exists \ a:=\limsup\limits_{|y|\rightarrow \infty}\frac{u(y)}{|y|}<\infty$
\end{remark}

\begin{claim}\label{claim 1}
$A\geq a.$
\end{claim}

\begin{proof}[Proof of the Claim] It suffices to prove $I:=\lbrace 0\leq \lambda \leq 1: A\geq \lambda a \rbrace=[0,1]$. If $a\leq 0$, there is nothing to prove. Assume $a>0$ and let $s:=\sup I$. Since $I$ is closed, $I=[0,s]$. It suffices to show that for any $\lambda \in I\cap[0,1)$, $\exists \ 0<\delta\ll 1$ such that $\lambda + \delta \in I.$

By the assumption $\widehat{H}(0)<\mu$, we can choose $w\in\mathcal{S}$, such that
\begin{equation*}
H(Dw,y)<\mu
\end{equation*}
Fix $R$, $\epsilon>0$, let $0<\delta<\min\lbrace \frac{\epsilon}{2a},1-\lambda\rbrace$, and define
\begin{equation*}
\phi_{R}(y):=\sqrt{R^{2}+|y|^{2}}-R
\end{equation*}
Define
\begin{equation*}
\widetilde{u}:=(\lambda+\delta)u+(1-\lambda-\delta)w \text{ and } \widetilde{v}:=v+\epsilon \phi_{R}
\end{equation*}
Then by level-set convexity,
\begin{eqnarray*}
H(D\widetilde{u},y)&=& H((\lambda+\delta)Du+(1-\lambda-\delta)Dw,y)\\
&\leq& \max \lbrace H(Du,y),H(Dw,y) \rbrace\\
&\leq& \mu
\end{eqnarray*}
By \textbf{(A4)}, if we choose $\epsilon\ll 1$, we have
\begin{equation*}
\mu<H(D\widetilde{v},y)
\end{equation*}
\begin{eqnarray*}
\liminf\limits_{|y|\rightarrow \infty}\frac{\widetilde{v}-\widetilde{u}}{|y|}&=&\liminf\limits_{|y|\rightarrow \infty}\frac{(v+\epsilon \phi_{R})-\big[(\lambda+\delta)u+(1-\lambda-\delta)w\big]}{|y|}\\
&=&\liminf\limits_{|y|\rightarrow \infty} \Bigg [ \frac{v-\lambda u}{|y|}+\epsilon \frac{\phi_{R}}{|y|}-\delta \frac{u}{|y|}-\frac{(1-\lambda-\delta)w}{|y|} \Bigg ]\\
&\geq& 0+\epsilon-\delta \cdot a-0\\
&\geq& \epsilon- \frac{\epsilon}{2a} \cdot a\\
&=&\frac{\epsilon}{2}
\end{eqnarray*}
So, $\widetilde{u}-\widetilde{v}$ attains its maximum in a bounded domain. And by the comparison principle in bounded domain.

\begin{equation*}
\widetilde{u}-\widetilde{v}\leq \max_{K}( \widetilde{u}-\widetilde{v} )
\end{equation*}

Let $R\rightarrow \infty$
\begin{equation*}
\widetilde{u}-v\leq \max_{K}( \widetilde{u}-v )
\end{equation*}

Which leads to 
\begin{equation}
(\lambda+\delta)u+(1-\lambda-\delta)w-v\leq \max\limits_{K}\big((\lambda+\delta)u+(1-\lambda-\delta)w-v\big)
\end{equation}
So,
\begin{eqnarray*}
\liminf\limits_{|y|\rightarrow \infty} \frac{v-(\lambda+\delta)u}{|y|}&=&\liminf\limits_{|y|\rightarrow \infty} \frac{v-[(\lambda+\delta)u+(1-\lambda-\delta)w]}{|y|}\\
&=&\liminf\limits_{|y|\rightarrow \infty} \frac{v-\widetilde{u}}{|y|}\\
&\geq&\liminf_{|y|\rightarrow \infty}\frac{-\min\limits_{K}(v-\widetilde{u})}{|y|}\\
&=& 0
\end{eqnarray*}

So, $I=[0,1]$, i.e. $A\geq a.$
 
\end{proof}

\begin{proof}[Proof of Lemma 2.3]
Since $I=[0,1]$, by argument similar to the one above, for any $\lambda\in[0,1)$
\begin{equation}
\lambda u+(1-\lambda)w-v\leq \max\limits_{K}\big(\lambda u+(1-\lambda)w-v\big)
\end{equation}

Letting $\lambda\rightarrow 1$, we get the lemma.
\end{proof}

\begin{remark}
The following statements give the homogenization of general level-set convex Hamiltonians under \textbf{(A1)-(A4)}.\\
\textbf{(I)} In [\ref{ref Armstrong and Souganidis}], the assumptions (\ref{additional assumptions(1)}), (\ref{additional assumptions(2)}) are only used to derive the comparison principle(Proposition 3.1 of [\ref{ref Armstrong and Souganidis}]).\\
\textbf{(II)} The above proof of the weak comparison principle does not require (\ref{additional assumptions(1)}), (\ref{additional assumptions(2)}).\\
\textbf{(III)} The weak comparison principle is sufficient to obtain the homogenization. Actually, in [1], the comparison principle is used mainly in two places. One is the construction of a maximal solution of the metric problem, the other is the proof of homogenization where the comparison principle is used to control the convergence in the approximated cell problem. Wherever we need the comparison principle, the above weak version comparison principle is sufficient.
\end{remark}

\section{The effect of strain term}\label{section 3}

We will study the following more general Hamiltonians with \textbf{(B1)-(B4)}.
\begin{equation}\label{cell problem with strain term}
H(p,x,\omega,c)=\sqrt{m^{2}+p^{2}}+cs(x,\omega)\cdot \frac{p}{\sqrt{m^{2}+p^{2}}}+k(x,\omega)
\end{equation}
\textbf{(B1)} Fix $\omega\in\Omega$, $k(x,\omega)$, $s(x,\omega)\in  C^{\infty}(\RR)$ and $k(x,\omega)$, $s(x,\omega)\in L^{\infty}(\RR\times \Omega).$\\ 
\textbf{(B2)} Fix $\omega\in\Omega$, if $k(x,\omega)$ achieves its local maximum value, then $s(x,\omega)=0.$\\ 
\textbf{(B3)} The event $\lbrace \omega\in\Omega: k(x,\omega) \text{ or } s(x,\omega) \text{ is constant}\rbrace$ is not of probability 1.\\
\textbf{(B4)} $k(x,\omega)$ and $s(x,\omega)$ are stationary. For any $x\in\RR$, $\mathbf{E}[s(x,\omega)]=0$. 

\begin{remark}\label{the remark with 7 components}

(1) The strain G-equation is a special case with $k=mv$ and $s=mv'$.\\

(2) We keep the existence of $\lbrace \tau_{z} \rbrace_{z\in\RR}$, which is ergodic. By this fact and \textbf{(B1)-(B4)}, without loss of generality, we can assume there are $\underline{k}<\overline{k}$, $\underline{s}<0<\overline{s}$, such that for all $\omega\in\Omega,$
\begin{equation*}
\inf\limits_{x\in\RR} k(x,\omega)=\underline{k},\ \sup\limits_{x\in\RR} k(x,\omega)=\overline{k},\
\inf\limits_{x\in\RR} s(x,\omega)=\underline{s},\ \sup\limits_{x\in\RR} s(x,\omega)=\overline{s}
\end{equation*}\\

(3) Since $H(\cdot,x,\omega,c)$ is level-set convex and $\min\limits_{p\in\RR}H(p,x,\omega,c)\leq |m|+\overline{k}$, there exist $ p_{-}(x,\omega,c)\leq p_{+}(x,\omega,c)$ and $p_{-}$, $p_{+}$ that are continuous functions of $x$ with
\begin{equation*}
\lbrace p: H(p,x,\omega,c)>|m|+\overline{k} \rbrace=(-\infty,p_{-}(x,\omega,c))\cup (p_{+}(x,\omega,c),\infty)
\end{equation*}\\

(4) By the homogenization result in [\ref{ref Armstrong and Souganidis}], the following is true: for a.e. $\omega\in\Omega$ and any $\delta>0$, if $u^{\delta}$ is the unique viscosity solution of
\begin{equation*}
\delta u^{\delta}+H(p+(u^{\delta})',x,\omega,c)=0, 
\end{equation*}
then we have
\begin{equation}\label{convergence in the minimum piece}
\lim\limits_{\delta\rightarrow 0}-\delta u^{\delta}(0,\omega)=\overline{H}(p,c).
\end{equation}
Without loss of generality, we can assume this statement is true for every $\omega\in\Omega$.\\

(5) It is easy to see $\overline{H}_{*}=\min\limits_{p\in\RR}\overline{H}(p,c)=|m|+\overline{k}$. By level-set convexity of $\overline{H}(p,c)$, there exist $\overline{p}_{-}(c)\leq \overline{p}_{+}(c)$ with
\begin{equation*}
\lbrace p: \overline{H}(p,c)>|m|+\overline{k}\rbrace=(-\infty,\overline{p}_{-}(c))\cup (\overline{p}_{+}(c),\infty)
\end{equation*}\\

(6) We will show (from Lemma \ref{P_{+}} to Lemma \ref{key}) that if $\overline{H}(p,c)>\overline{H}_{*}$, for fixed $\omega,c$, the following cell problem has a sub-linear solution $\gamma(x,\omega,c)$.
\begin{equation}\label{The equation of cell problem}
H(p+\gamma',x,\omega,c)=\overline{H}(p,c)
\end{equation}\\

(7) Cell problems (\ref{The equation of cell problem}) do not have solutions (see [\ref{Lions Souganidis corrector}]) in general. Here in the 1-dimensional level-set convex setting, the above remark (6) says for those $\overline{H}(p,c)>\overline{H}_{*}$, cell problems do have solutions. More generally, in the 1-dimensional coercive situation, if $\overline{H}(p)$ is not a local extreme value, the solution of the cell problem at $p\in\RR$ always exists (see [\ref{1-d seperable noncovex by ATY}]). As for those $\overline{H}(p,c)=\overline{H}_{*}$, the identity (\ref{convergence in the minimum piece}) can be obtained by using comparison principle (see [\ref{1-d seperable noncovex by ATY}]).  \\

(8) From Lemma \ref{P_{+}} to Lemma \ref{key}, we always fix $c$. In Theorem \ref{main theorem}, we fix $p=(m,n)$ and study how $\overline{H}(p,c)$ depends on $c$.

\end{remark}
\begin{lemma}\label{P_{+}}
Fix $c\in[0,\infty)$. For any $\mu\in(\overline{H}_{*},\infty)$, there exists a unique $ P_{+}(\mu,c)$, such that for each $\omega$, the equation
\begin{equation*}
\begin{cases}
H(P_{+}(\mu,c)+\gamma'(x,\omega,c),x,\omega,c)=\mu \\ P_{+}(\mu,c)+\gamma'(x,\omega,c)> p_{+}(x,\omega,c)
\end{cases}
\end{equation*}

admits a viscosity solution $\gamma(x,\omega,c)$ and for a.e. $\omega\in\Omega$, $\gamma(x,\omega,c)$ is sub-linear.
\end{lemma}
\begin{proof}
For each $\mu>\overline{H}_{*}$, consider the equation
\begin{equation*}
H(u'(x,\omega,c),x,\omega,c)=\mu
\end{equation*}

By the fact that $\min\limits_{p\in\RR}H(p,x,\omega,c)\leq \overline{H}_{*}$ and $H(p,x,\omega,c)$ is strictly level set convex. There are exactly two solutions of $u'(x,\omega,c)$, one is less than $p_{-}(x,\omega,c)$, the other is greater than $p_{+}(x,\omega,c)$. We choose the latter one, by stationary of $H$, $u'(x,\omega,c)$ is stationary. By smoothness of $H(\cdot,\cdot,\omega,c)$ and $\mu>\overline{H}_{*}$, $u'(x,\omega,c)$ is smooth with respect to $x$, so by continuity, we always have that $u'(x,\omega,c)>p_{+}(x,\omega,c)$.

Since $H(p,x,\omega,c)$ is coercive with respect to $p$, uniformly with respect to $x\in\RR$, $u'(x,\omega,c)$ is bounded. We can define
\begin{equation*}
P_{+}(\mu,c):=\mathbf{E}[u'(x,\omega,c)]
\end{equation*}

Due to the stationary of $u'$, the expectation is independent of $x$ and is uniquely defined for each $c\geq 0$ and $\mu>\overline{H}_{*}.$

Then we define the function
\begin{equation*}
\gamma(x,\omega,c):=u(x,\omega,c)-P_{+}(\mu,c)\cdot x
\end{equation*}

Then $\mathbf{E}[\gamma'(x,\omega,c)]=0$ and by sub-additive Ergodic Theorem, for a.e. $\omega\in\Omega,$
\begin{equation*}
\lim\limits_{|x|\rightarrow \infty}\frac{\gamma(x,\omega,c)-\gamma(0,\omega,c)}{|x|}=\lim\limits_{|x|\rightarrow \infty}\frac{1}{|x|}\int_{0}^{x}\gamma'(s,\omega,c)ds=\mathbf{E}[\gamma'(0,\omega,c)]=0
\end{equation*}

Hence for a.e. $\omega\in\Omega,$ $\gamma(x,\omega,c)$ is sub-linear and this is the desired solution.
\end{proof}

By the same argument, we have:

\begin{lemma}
Fix $c\in[0,\infty)$. For any $\mu\in(\overline{H}_{*},\infty)$, there exists a unique $ P_{-}(\mu,c)$, such that for each $\omega$, the equation
\begin{equation*}
\begin{cases}
H(P_{-}(\mu,c)+\gamma'(x,\omega,c),x,\omega,c)=\mu \\ P_{-}(\mu,c)+\gamma'(x,\omega,c)< p_{-}(x,\omega,c)
\end{cases}
\end{equation*}

admits a viscosity solution $\gamma(x,\omega,c)$ and for a.e. $\omega\in\Omega$, $\gamma(x,\omega,c)$ is sub-linear.
\end{lemma}

\begin{proposition}
Fix $c\in[0,+\infty)$. The function $P_{+}(\mu,c):(\overline{H}_{*},+\infty)\longrightarrow \RR$ has the following properties:

$(1)$ $P_{+}(\mu,c)$ is strictly increasing.

$(2)$ $P_{+}(\mu,c)$ is continuous.

$(3)$ $\lim\limits_{\mu\rightarrow +\infty}P_{+}(\mu,c)=+\infty$.

\end{proposition}

\begin{proof}
(1) Since $H(p,x,\omega,c)$, as a function of $p$, is strictly increasing on $(p_{+}(x,\omega,c),+\infty)$, and it's uniformly continuous. So
\begin{equation*}
\overline{H}_{*}<\mu_{1}<\mu_{2}<\infty\implies P_{+}(\mu_{1},c)<P_{+}(\mu_{2},c)
\end{equation*}

(2) Suppose $\mu_{n}$, $\mu\in(\overline{H}_{*},+\infty)$ and $\mu_{n}\rightarrow \mu$ as $n\rightarrow \infty$. Accordingly, we can solve $u'_{n}$ and $u'$ by Lemma \ref{P_{+}}.
\begin{eqnarray*}
H(u_{n}'(x,\omega,c),x,\omega,c)&=&\mu_{n}\\
H(u'(x,\omega,c),x,\omega,c)&=&\mu\\
\end{eqnarray*}

For each fixed $(x,\omega,c)\in\RR\times\Omega\times \RR^{+}$, by the fact that $H(\cdot,x,\omega,c)$ is smooth and strictly increasing on $(p_{+}(x,\omega,c),+\infty)$, we have $\lim\limits_{n\rightarrow \infty}u_{n}'(x,\omega,c)=u'(x,\omega,c)$, by bounded convergence theorem, $\lim\limits_{n\rightarrow\infty}\mathbf{E}[u_{n}'(x,\omega,c)]=\mathbf{E}[u'(x,\omega,c)]$. 

Thus, $\lim\limits_{n\rightarrow\infty}P_{+}(\mu_{n},c)=P_{+}(\mu,c).$

(3) If $P_{+}(\mu,c)$ is bounded, since $k(x,\omega)$, $s(x,\omega)$ are uniformly bounded and then $H(P_{+}(\mu,c),x,\omega,c)$ is uniformly bounded.
Let $\mathbf{E}[\gamma'(x,\omega)]=0$ and $\gamma(x,\omega)$ solves
\begin{equation*}
H(P_{+}(\mu,c)+\gamma',x,\omega,c)=\mu
\end{equation*}

Since $\gamma(x,\omega)$ is sub-linear and smooth. For any $\epsilon>0$, there is some interval $(a(\mu),b(\mu))$ on which $|\gamma'|<\epsilon$(Otherwise, by continuity, $\gamma$ will be at least linear growth at infinity). So $H(P_{+}(\mu,c)+\gamma',x,\omega,c)$ is uniformly bounded on $(a(\mu),b(\mu))$, this gives a contradiction when $\mu\rightarrow +\infty.$
\end{proof}

Similarly, we can prove:
\begin{proposition}
Fix $c\in[0,+\infty)$. The function $P_{-}(\mu,c):(\overline{H}_{*},+\infty)\longrightarrow \RR$ has the following properties:

$(1)$ $P_{-}(\mu,c)$ is strictly decreasing.

$(2)$ $P_{-}(\mu,c)$ is continuous.

$(3)$ $\lim\limits_{\mu\rightarrow +\infty}P_{-}(\mu,c)=-\infty$.

\end{proposition}

\begin{definition}
By the above propositions of $P_{+}(\mu,c)$ and $P_{-}(\mu,c)$, we denote their inverse functions by $\mu_{+}(p,c)$ and $\mu_{-}(p,c)$.
\begin{equation*}
\mu_{+}(p,c):(\inf\limits_{\mu} P_{+}(\mu,c),+\infty)\longrightarrow (\overline{H}_{*},+\infty)
\end{equation*}

\begin{equation*}
\mu_{-}(p,c):(-\infty, \sup\limits_{\mu} P_{-}(\mu,c))\longrightarrow (\overline{H}_{*},+\infty)
\end{equation*}

And then we can define the continuous level-set convex function $\mu(p,c)$.
\begin{equation*}
\mu(p,c):=
\begin{cases}
\mu_{-}(p,c)  &  \text{ if } p\in (-\infty, \sup\limits_{\mu} P_{-}(\mu,c))\\
\overline{H}_{*} &  \text{ if } p\in [\sup\limits_{\mu} P_{-}(\mu,c),\inf\limits_{\mu} P_{+}(\mu,c)]\\
\mu_{+}(p,c)  &  \text{ if } p\in (\inf\limits_{\mu} P_{+}(\mu,c),+\infty)
\end{cases}
\end{equation*}
\end{definition}

\begin{lemma}\label{key}
$\mu(p,c)=\overline{H}(p,c)$
\end{lemma}

\begin{proof}

By the existence of cell problem, $\mu(p,c)=\overline{H}(p,c)$, $\forall p\in(-\infty,\sup\limits_{\mu} P_{-}(\mu,c))\cup(\inf\limits_{\mu} P_{+}(\mu,c),+\infty)$. By level-set convexity of $\overline{H}(p,c)$ and $\overline{H}(p,c)\geq \overline{H}_{*}$, we have $\overline{H}(p,c)=\overline{H}_{*}$, $\forall p\in[\sup\limits_{\mu} P_{-}(\mu,c),\inf\limits_{\mu} P_{+}(\mu,c)]$. So $\mu(p,c)=\overline{H}(p,c)$.
\end{proof}

The next theorem is aimed to study the dependence of $\overline{H}(n,c)$ on $c$. As mentioned under (\ref{1-d problem}), $\overline{H}(n,c)$ is equal to $\overline{H}(m,n,c)$ in the original 2-d problem. We will fix a unit vector $(m,n)\in\RR^{2}$ and denote $h(c):=\overline{H}(n,c)=\overline{H}(m,n,c)$ in the following theorem.

\begin{theorem}\label{main theorem}Under \textbf{(B1)-(B4)}, fix a unit vector $(m,n)\in\RR^{2}$ with $mn\ne 0$.

$(1)$ $h(c)\in C^{0,1}(\RR^{+})$ and $\Vert s\Vert:=\Vert s(x,\omega)\Vert_{L^{\infty}(\RR \times \Omega)}$ is the Lipschitz constant. 

$(2)$ $h'(c)\leq 0$ for a.e. $c\in(0,\infty)$. If $h(c)>\overline{H}_{*}$, $h'(c)<0$.

$(3)$ There exists $\overline{c}>0$, when $c>\overline{c}$, $h(c)=\overline{H}_{*}$.
\end{theorem}

\begin{proof}
(1) Fix $c_{1}$, $c_{2}\in(0,\infty)$, then $\overline{H}_{*}\leq h(c_{1})$, $h(c_{2})<\infty$. For each $0<\delta\ll 1$, let $u^{\delta}$, $v^{\delta}$ be the unique solutions of the following two equations.
\begin{equation}\label{app c1}
\delta u^{\delta}+\sqrt{m^{2}+(n+(u^{\delta})')^{2}}+\frac{c_{1}(n+(u^{\delta})')s(x,\omega)}{\sqrt{m^{2}+(n+(u^{\delta})')^{2}}}+k(x,\omega)= 0
\end{equation}

\begin{equation}\label{app c2}
\delta v^{\delta}+\sqrt{m^{2}+(n+(v^{\delta})')^{2}}+\frac{c_{2}(n+(v^{\delta})')s(x,\omega)}{\sqrt{m^{2}+(n+(v^{\delta})')^{2}}}+k(x,\omega)= 0
\end{equation}

By Remark \ref{the remark with 7 components}.

\begin{equation*}
\lim\limits_{\delta\rightarrow 0}|\delta u^{\delta}(0,\omega)+h(c_{1})|=\lim\limits_{\delta\rightarrow 0}|\delta v^{\delta}(0,\omega)+h(c_{2})|=0
\end{equation*}

Since $w(x,\omega):=v+\frac{1}{\delta}\Vert s \Vert |c_{2}-c_{1}|$ is a super solution of (\ref{app c1}). By comparison principle, $\delta u^{\delta}\leq \delta v^{\delta}+\Vert s \Vert |c_{2}-c_{1}|$, similarly, $\delta v^{\delta}\leq \delta u^{\delta}+\Vert s \Vert |c_{2}-c_{1}|$, thus
\begin{equation*}
|\delta u^{\delta}(0,\omega)-\delta v^{\delta}(0,\omega)|\leq \Vert s \Vert |c_{2}-c_{1}|
\end{equation*}

Let $\delta \rightarrow 0$, we get
\begin{equation*}
|h(c_{2})-h(c_{1})|\leq \Vert s \Vert|c_{2}-c_{1}| 
\end{equation*}

(2) Fix $c_{0}>0$. Without of loss of generality, we assume $n>0$. Lipschitz function is differentiable a.e., so if $h(c_{0})=\overline{H}_{*}$ and $h(c_{0})$ is differentiable at $c_{0}$, then $h'(c_{0})=0$. 

Now assume $h(c_{0})> \overline{H}_{*}$ and denote $f(t):=\sqrt{m^{2}+t^{2}}.$ 

We will focus on the cell problem $H(n+u',x,\omega,c)=h(c).$

By continuity, there is some $\epsilon>0$, such that for $c\in I_{\epsilon}=(c_{0}-\epsilon,c_{0}+\epsilon)\cap\RR^{+}$, $h(c)-\overline{H}_{*}$ has a positive lower bound. Since  $u(x,\omega,c)$ is smooth, $n+u'(x,\omega,c)>0$ has a positive lower bound.

To show $h'(c_{0})<0$, we first show:

\textbf{Claim 1:}
For $c\in I_{\epsilon}$ (here $t=n+u'$ in $f(t)$).
\begin{equation}\label{claim 1 in the proof of main theorem}
f'+cs(x,\omega)f''=\frac{(n+u')^{3}+m^{2}(n+u')+cs(x,\omega)m^{2}}{(m^{2}+(n+u')^{2})^{\frac{3}{2}}}>0
\end{equation}

To prove \textbf{Claim 1}, it suffices to show $(n+u')+cs(x,\omega)>0$.

By the fact that $\overline{H}_{*}=|m|+\overline{k}$ and
\begin{equation*}
\sqrt{m^{2}+(n+u')^{2}}+\frac{c(n+u')s(x,\omega)}{\sqrt{m^{2}+(n+u')^{2}}}+k(x,\omega)=h(c)>\overline{H}_{*}
\end{equation*}

We have
\begin{equation*}
\sqrt{m^{2}+(n+u')^{2}}-|m|+\frac{c(n+u')s(x,\omega)}{\sqrt{m^{2}+(n+u')^{2}}}>0
\end{equation*}

This is equivalent to
\begin{equation*}
\frac{(n+u')^{2}}{\sqrt{m^{2}+(n+u')^{2}}+|m|}+\frac{c(n+u')s(x,\omega)}{\sqrt{m^{2}+(n+u')^{2}}}>0
\end{equation*}

So
\begin{equation*}
\frac{(n+u')^{2}+c(n+u')s(x)}{\sqrt{m^{2}+(n+u')^{2}}}=\frac{(n+u')^{2}}{\sqrt{m^{2}+(n+u')^{2}}}+\frac{c(n+u')s(x,\omega)}{\sqrt{m^{2}+(n+u')^{2}}}>0
\end{equation*}

Which means
\begin{equation*}
(n+u')(n+u'+cs(x,\omega))>0
\end{equation*}

The fact that $n+u'>0$ implies $n+u'+cs(x,\omega)>0$. Thus \textbf{Claim 1} is proved.\\

Immediately, we have:
\begin{equation}\label{positive expectation}
\mathbf{E}\Big[\frac{1}{f'+cs(x,\omega)f''}\Big]>0
\end{equation}

The fact that 
\begin{equation*}
f'+cs(x,\omega)f''=\frac{(n+u')^{3}+m^{2}(n+u')+cs(x,\omega)m^{2}}{(m^{2}+(n+u')^{2})^{\frac{3}{2}}}>\frac{(n+u')^{3}}{(m^{2}+(n+u')^{2})^{\frac{3}{2}}}
\end{equation*}

implies $f'+cs(x,\omega)f''$ has a positive lower bound. And the fact that
\begin{equation*}
f'(n+u')=\frac{n+u'}{\sqrt{m^{2}+\big(n+u'\big)^{2}}}
\end{equation*}

implies $f'(n+u')$ has a positive lower bound. 

If we denote
\begin{equation*} 
a(x,\omega,c):=\frac{cf''}{f'}=\frac{m^{2}c}{(n+u')(m^{2}+(n+u')^{2})}> 0
\end{equation*}

Then by dividing $f'$ in (\ref{claim 1 in the proof of main theorem}).
\begin{equation*}
1+a(x,\omega,c)s(x,\omega)>0 \text{ has a positive lower bound}
\end{equation*}

Now, the cell problem can be rewritten as (\ref{rewritten eqns}). Since $F(t):=f(t)+csf'(t)+k$ is smooth and increasing with respect to $t=n+u'(x,\omega,c)$ and $h(c)\in C^{0,1}(\RR^{+})$, $u'=F^{-1}(h(c))-n$ is differentiable a.e. with respect to $c$.
\begin{equation}\label{rewritten eqns}
f(n+u'(x,\omega,c))+cs(x,\omega)f'(n+u'(x,\omega,c))+k(x,\omega)=h(c)
\end{equation}

Differentiate it w.r.t. $c$ gives: \Big(here $\frac{\partial}{\partial c}(u')=\frac{\partial}{\partial c}(\frac{\partial}{\partial x}u(x,\omega,c))$ \Big) 
\begin{equation}\label{diff cell rewrite}
h'(c)\cdot\frac{1}{f'+cs(x,\omega)f''}=\frac{s(x,\omega)}{1+a(x,\omega,c)s(x,\omega)}+\frac{\partial}{\partial c}(u') 
\end{equation}

The above positive lower bounds as well as the boundedness of $h'$ and $s(x,\omega)$ implies that $\frac{\partial}{\partial c}(u')$ is bounded uniformly for $(c,x,\omega)\in I_{\epsilon}\times\RR\times\Omega$. This will allow us to apply bounded convergence theorem in (\ref{BDD CNV THM}).

Taking expectation in (\ref{diff cell rewrite}) gives:
\begin{equation*}
h'(c_{0})\cdot\mathbf{E}\Big[\frac{1}{f'+c_{0}s(x,\omega)f''}\Big]=\mathbf{E}\Big[\frac{s(x,\omega)}{1+a(x,\omega,c_{0})s(x,\omega)}\Big]+\mathbf{E}\Big[\frac{\partial}{\partial c}(u')(x,\omega,c_{0})\Big] 
\end{equation*}

Choose $I_{\epsilon}\ni c_{k}\rightarrow c_{0}$, by bounded convergence theorem and the fact $\mathbf{E}[u']=0$.
\begin{eqnarray}\label{bdd conv thm}
\mathbf{E}\Big[\frac{\partial}{\partial c}(u')(x,\omega,c_{0})\Big]&=&\mathbf{E}\Big[\lim\limits_{I_{\epsilon}\ni c_{k}\rightarrow c_{0}}\frac{u'(x,\omega,c_{k})-u'(x,\omega,c_{0})}{c_{k}-c_{0}}\Big]\\
&=& \lim\limits_{I_{\epsilon}\ni c_{k}\rightarrow c_{0}} \mathbf{E}\Big[\frac{u'(x,\omega,c_{k})-u'(x,\omega,c_{0})}{c_{k}-c_{0}}\Big]\label{BDD CNV THM}\\
&=& 0
\end{eqnarray}

Recall that $a(x,\omega,c_{0})>0$, $1+a(x,\omega,c_{0})s(x,\omega)>0$ and $s(x,\omega)$ is not a constant function, we have
\begin{eqnarray*}
\mathbf{E}\Big[\frac{s(x,\omega)}{1+a(x,\omega,c_{0})s(x,\omega)}\Big]&=& \mathbf{E}\Big[\frac{s(x,\omega)}{1+a(x,\omega,c_{0})s(x,\omega)}\chi_{\lbrace\omega:s(x,\omega)>0\rbrace}\Big]\\
&+&\mathbf{E}\Big[\frac{s(x,\omega)}{1+a(x,\omega,c_{0})s(x,\omega)}\chi_{\lbrace\omega:s(x,\omega)\leq 0\rbrace}\Big]\\
&<& \mathbf{E}\Big[s(x,\omega)\chi_{\lbrace\omega:s(x,\omega)>0\rbrace}\Big]+\mathbf{E}\Big[s(x,\omega)\chi_{\lbrace\omega:s(x,\omega)\leq 0\rbrace}\Big]\\
&=& \mathbf{E}[s(x,\omega)]\\
&=& 0
\end{eqnarray*}

Combine these with (\ref{positive expectation}), we can conclude:
\begin{equation*}
h'(c_{0})<0
\end{equation*}

(3) Without loss of generality, let $n>0$ and $\tau:=\vert\underline{s}\vert=-\underline{s}>0$.

For each $\omega\in\Omega$, there are countable disjoint intervals $\lbrace (l_{i}(\omega),r_{i}(\omega)): r_{i-1}\leq l_{i}, i\in\ZZ\rbrace$ such that
\begin{equation*}
A(\omega):=\lbrace x:s(x,\omega)< -\frac{\tau}{2}\rbrace=\bigcup\limits_{i\in\ZZ}\big(l_{i}(\omega),r_{i}(\omega)\big)
\end{equation*}

Denote
\begin{equation*}
B(\omega):=\RR\diagdown A(\omega)=\bigcup\limits_{i\in\ZZ}\big[r_{i}(\omega),l_{i+1}(\omega)\big]
\end{equation*}
Since $s(x,\omega)<-\frac{\tau}{2}$ on $A(\omega)$, by the stationary of $\chi_{A(\omega)}(x)$ and $\chi_{B(\omega)}(x)$ we have for a.e. $\omega\in\Omega$:
\begin{equation*}
\alpha:=\lim\limits_{L\rightarrow+\infty}\frac{1}{2L}\int_{-L}^{L}\chi_{A(\omega)}(x)dx=\mathbf{P}\Big[\omega\in\Omega: s(0,\omega)<-\frac{\tau}{2}\Big] 
\end{equation*}

By \textbf{(B4)} and Remark \ref{the remark with 7 components}, $\alpha\in(0,1)$. Now we can construct a smooth stationary function $\psi(x,\omega)$ with $\psi(x,\omega)=0$ on $B(\omega)$ and 

\begin{equation*}
\frac{1}{r_{i}-l_{i}}\int_{l_{i}}^{r_{i}}\psi(x,\omega)dx=\frac{n}{\alpha} \text{ and } 0\leq\psi(x,\omega)\leq \frac{2n}{\alpha}
\end{equation*}

Then we will have
\begin{equation*}
\lim\limits_{L\rightarrow+\infty}\frac{1}{2L}\int_{-L}^{L}\psi(x,\omega)dx=\lim\limits_{L\rightarrow+\infty}\frac{1}{L}\int_{0}^{L}\psi(x,\omega)dx=\lim\limits_{L\rightarrow+\infty}\frac{1}{L}\int_{-L}^{0}\psi(x,\omega)dx=n
\end{equation*}

Let $\phi'(x,\omega):=\psi(x,\omega)-n$, then

\begin{equation*}
\lim\limits_{|x|\rightarrow \infty}\frac{\phi(x,\omega)-\phi(0,\omega)}{|x|}=\lim\limits_{|x|\rightarrow \infty}\frac{1}{|x|}\int_{0}^{x}\phi'(s,\omega)ds=0
\end{equation*}

Which means that for a.e. $\omega\in\Omega$, $\phi(x,\omega)$ is sub-linear.\\

The derivative of $g(t):=\sqrt{m^{2}+t^{2}}+\frac{cts(x,\omega)}{\sqrt{m^{2}+t^{2}}}$ with respect to $t$ is $\frac{t^{3}+m^{2}t+m^{2}cs(x,\omega)}{(m^{2}+t^{2})^{\frac{3}{2}}}$,
Let $\overline{c}:=\frac{2}{\tau m^{2}}\big [(\frac{2n}{\alpha})^{3}+m^{2}(\frac{2n}{\alpha})\big ]>0$. For all $x\in A(\omega)$, if $c>\overline{c}$, then $g'(t)<0$ for $t\in[0,\frac{2n}{\alpha}]$. By the construction of $\phi$,
\begin{equation*}
0\leq n+\phi'(x,\omega)=\psi(x,\omega)\leq \frac{2n}{\alpha}
\end{equation*}

And recall that $\text{supp}(n+\phi'(x,\omega))\subset A(\omega)$, then

\begin{eqnarray*}
\max\limits_{x\in\RR}\Big\lbrace \sqrt{m^{2}+(n+\phi')^{2}}+\frac{c(n+\phi')s}{\sqrt{m^{2}+(n+\phi')^{2}}}+k \Big\rbrace 
&\leq&\max\limits_{x\in\RR} \lbrace |m|+k(x,\omega) \rbrace\\
&=& |m|+\max_{x\in\RR}k(x,\omega)\\
&=& \overline{H}_{*}
\end{eqnarray*}

If $h(c)>\overline{H}_{*}$, by Lemma \ref{key}, the cell problem has solution $u(x,\omega)$ which is sub-linear for a.e. $\omega\in\Omega$. By above construction, $\phi$ is also sub-linear for a.e. $\omega\in\Omega$. Fix such $\omega$ that both of $\phi(x,\omega)$ and $u(x,\omega)$ are sub-linear. So for any $\delta>0$, $u(x,\omega)-\phi(x,\omega)+\delta\sqrt{x^{2}+1}$ can achieve minimum at some point $x_{\delta}$, so
\begin{equation*}
h(c)\leq H(n+\phi'(x_{\delta},\omega)-\delta \frac{x_{\delta}}{\sqrt{x_{\delta}^{2}+1}},x_{\delta},\omega)
\end{equation*}

$\delta\rightarrow 0 \implies h(c)\leq \max\limits_{x\in\RR} H(n+\phi'(x,\omega),x,\omega)=\overline{H}_{*}$, this is a contradiction.

Thus $h(c)=\overline{H}_{*}$ when $c>\overline{c}.$
\end{proof}

\begin{proof}[Proof of theorem \ref{main thm}]
(1) comes from section 2.

(2) If $mn\ne 0$, by Theorem \ref{main theorem} with $k(x,\omega)=mv(x,\omega)$, $s(x,\omega)=mv'(x,\omega)$.

If $m=0$, $\overline{H}(p)=\overline{H}(p,c)=|n|=1>0=|m|+\sup\limits_{(x,\omega)\in\RR\times\Omega}mv(x,\omega)$.

If $n=0$, $\overline{H}(p)=\overline{H}(p,c)=|m|+\sup\limits_{(x,\omega)\in\RR\times\Omega}mv(x,\omega)$.

(3) If $mv\equiv 0$, then $\overline{H}(p)\equiv \overline{H}(p,c)$.\\

Suppose $\overline{H}(m,n)=\overline{H}(m,n,c)>\overline{H}_{*}.$

If $mn\ne 0$ we must have $v$ is constant, otherwise by Theorem \ref{main theorem}, $\overline{H}(m,n,c)>\overline{H}(m,n)$ which gives a contradiction. By $\mathbf{E}[v]=0$, we must have $mv=0.$

If $m=0$ then $mv\equiv 0.$

If $n=0$, this is impossible since $\overline{H}(m,n)=\overline{H}(m,n,c)\equiv\overline{H}_{*}.$

Thus $mv\equiv 0.$
\end{proof}

\section{Acknowledgements}
The author would like to thank his advisor Yifeng Yu for his generous support and providing the topic. The author is also grateful to Andrew J. Thomas and Christopher Lopez for their helpful suggestions in writing the paper. Particularly, the author want to thank the referee in providing many constructive revision suggestions to improve the presentation of this article.

\bibliographystyle{amsplain}

\end{document}